\documentclass{amsart}
\usepackage{amsmath, amsthm,upref, amssymb,pdfsync}
\usepackage{blindtext}
\usepackage{enumitem}
 \usepackage{amsmath,amscd,amsthm}
\usepackage{mathrsfs}
\usepackage{tikz}
\input xypic
\xyoption{all}

\theoremstyle{plain} 
\newtheorem{theorem}[subsection]{Theorem}
\newtheorem{proposition}[subsection]{Proposition}
\newtheorem{lemma}[subsection]{Lemma}
\newtheorem{corollary}[subsection]{Corollary}

\theoremstyle{definition}

\newtheorem{example}[subsection]{Example}
\newtheorem{definition}[subsection]{Definition}

\theoremstyle{remark}
\newtheorem{remark}[subsection]{Remark}
\newtheorem*{remark*}{Remark}

\numberwithin{equation}{subsection}

\DeclareMathOperator{\cone}{{\mathrm{cone}}}

\usepackage{color}
\definecolor{Mcolor}{rgb}{0,0,1}
\definecolor{Wcolor}{rgb}{1,0,0}
\definecolor{lightgray}{rgb}{0.6,0.6,0.6}

\begin{document}

\title[Twisted Weak Orders of Coxeter Groups]{Twisted Weak Orders of Coxeter Groups}

\author{Weijia Wang}
\address{School of Mathematics (Zhuhai)
\\ Sun Yat-sen University \\
Zhuhai, Guangdong, 519082 \\ China}
\email{wangweij5@mail.sysu.edu.cn}
\date{\today}

\begin{abstract}
In this paper, we initiate the study of the twisted weak order associated to a twisted Bruhat order for a Coxeter group and explore the relationship between the lattice property of such orders and the infinite reduced words. We show that for a 2 closure biclosed set $B$ in $\Phi^+$, the $B$-twisted weak order is a non-complete meet semilattice if $B$ is the inversion set of an infinite reduced word and that the converse also holds in the case of affine Weyl groups.
\end{abstract}

\maketitle

\section{Introduction}

Twisted Bruhat order on a Coxeter group was first introduced by Dyer. Such an order generalizes the usual Bruhat order. They are of interest for their own combinatorics and have  connections (some conjecturally) with the reflection orders, the geometries of the root systems and the representation theory. For example, the alcove order of the affine Weyl groups studied by Lusztig is a twisted Bruhat order and in \cite{DyHS2} Kazhdan-Lusztig polynomial is defined for certain intervals of the twisted Bruhat order and therefore such an order relates  to the composition factors of the Verma modules. In \cite{Gobet},  twisted Bruhat orders and  biclosed sets of the positive roots are used in the study of the twisted filtration of the Soergel bimodules and the generalized positivity conjecture of the Hecke algebra.

To every ordinary Bruhat order one can associate a weak order, which encodes the important combinatorics of the underlying Coxeter group.
 In the same spirit, we construct the natural twisted weak order associated to any given twisted Bruhat order. Such an order is a generalization of the ordinary weak order and occurs naturally in many constructions, e.g. limit weak order of the infinite reduced words.
We show that several definitions of the twisted weak order can be given and establish their equivalence.
 The main result of the paper deals with the relationship among the infinite reduced words, the geometries of the root systems and the lattice property of the twisted weak order.
We prove that if a biclosed set (resp. biconvex set) in $\Phi^+$ is the inversion set of an infinite reduced word then the corresponding twisted weak order is a non-complete lattice and for the affine Weyl groups the converse also holds. Finding a characterization of the biclosed sets which arise from the inversion sets is an interesting question and is discussed in \cite{Inversion}. Our result provides a new perspective which uses the twisted weak order.

\section{Preliminaries}

\subsection{Closure Operators on The Root Systems}

We refer the readers to \cite{bjornerbrenti} for the basic notions of Coxeter groups and the root systems. With the canonical bijection between the simple reflections and the simple roots, we denote the simple root corresponding to the simple reflection $s$ by $\alpha_s$. We denote the set of simple reflections by $S$ and the set of reflections by $T$. In this paper, we consider two types of closure operators: 2 closure operator and cone closure operator on the root system $\Phi$ of a Coxeter group $W$.
The 2 closure operator is studied for its relevance with reflection orders, description of the join and the meet under the weak order and the completion of the weak order. A subset $\Gamma$ of $\Phi$ is said to be 2 closure closed if for any $\alpha,\beta\in \Gamma$, the set $\{k_1\alpha+k_2\beta|k_1,k_2\in \mathbb{R}_{\geq 0}\}\cap \Phi$ is still contained in $\Gamma$. Under this operator, the closure of any set is the intersection of all 2 closure closed sets containing it.
Cone closure operator is the natural oriented matroidal closure operator on $\Phi$. A subset $\Gamma$ of $\Phi$ is said to be cone closure closed if for any finite number of roots $\alpha_1,\alpha_2,\cdots,\alpha_t\in \Gamma$, the set
$\{\sum_{i=1}^tk_i\alpha_i|k_i\in \mathbb{R}_{\geq 0}\}\cap \Phi$ is still contained in $\Gamma$. Similarly the cone closure of any set is the intersection of all cone closure closed sets containing it.

 A set $\Lambda\subset \Gamma(\subset \Phi)$ with the property that both $\Lambda$ and $\Gamma\backslash \Lambda$ are 2 closure closed (resp. cone closure closed) is called a 2 closure biclosed set (resp. cone closure biclosed set) in $\Gamma$. Clearly a 2 closure biclosed set in $\Gamma$ is also cone closure biclosed but not vice versa. A 2 closure biclosed (resp. cone closure) set $H$ in $\Phi$ such that $\Phi\backslash H=-H$ is called a 2 closure hemispace (resp. cone closure hemispace).
For example $\Phi^+$ is both a 2 closure hemispace and a cone closure hemispace.
One can easily verify that there exists a bijection between the set of  2 closure biclosed sets in $\Phi^+$ and the set of 2 closure hemispaces. To make it explicit, given a 2 closure biclosed set $B$ in $\Phi^+$ we can associate a 2 closure hemispace $B\uplus -(\Phi^+\backslash B)$ where $\uplus$ denotes the disjoint union. (To see that $B\uplus -(\Phi^+\backslash B)$ is indeed 2 closure biclosed in $\Phi$, it suffices to show that it is 2 closure closed in $\Phi$ as its complement in $\Phi$ is $(\Phi^+\backslash B)\uplus -B$. Note that $B$ and $-(\Phi^+\backslash B)$ are both closed since $B$ is biclosed in $\Phi^+$. Now take $\alpha\in B$ and $\beta\in \Phi^+\backslash B$. If $k_1\alpha-k_2\beta=\gamma\in -B, k_1, k_2\geq 0$, then $\beta=\frac{k_1}{k_2}\alpha-\frac{1}{k_2}\gamma\in B$ as $B$ is closed. A contradiction. If $k_1\alpha-k_2\beta=\gamma\in \Phi^+\backslash B, k_1,k_2>0$, then $\frac{1}{k_1}\gamma+\frac{k_2}{k_1}\beta=\alpha\in \Phi^+\backslash B$ as $\Phi^+\backslash B$ is closed. A contradiction.) Given a 2 closure hemispace $H$ we can associate to it a 2 closure biclosed set in $\Phi^+$: $H\cap \Phi^+.$ In fact such a bijection also exists if one replace $\Phi^+$ with any 2 closure hemispace $H$.  But there is in general no such bijection for the cone closure operator. (There could exist a cone closure biclosed set $B$ in $\Phi^+$ such that $B\uplus -(\Phi^+\backslash B)$ is not cone closure closed.) There is a natural $W$ action on the set of 2 closure hemispaces (i.e. $wH=\{w(\alpha)|\alpha\in H\}$ for $H$ a hemispace) and thus can be transferred to the set of all 2 closure biclosed sets in $\Phi^+$. We denote such an action by $w\cdot B$ for $w\in W$ and $B$ a biclosed set in $\Phi^+$ (The formula of this action is given by $w\cdot B=(\Phi_w\backslash w(-B))\cup (w(B)\backslash -\Phi_w)$ where $\Phi_w:=\{\alpha\in \Phi^+|w^{-1}(\alpha)\in \Phi^-\}$ is called an inversion set). Note that such an action is certainly different from the natural $W$ action on a set $B$ which we denote $wB$. A cone closure biclosed set in $\Phi^+$ is also called a biconvex set in $\Phi^+$ in the literature.

It is known that any finite 2 closure (resp. cone closure) biclosed set in $\Phi^+$ is an inversion set $\Phi_x$ for some $x\in W$. For a proof see \cite{DyerWeakOrder} Lemma 4.1(d). Note that if $s_1s_2\cdots s_k$ is a reduced expression of $x$, then $\Phi_x=\cup_{i=1}^k\Phi_{s_1s_2\cdots s_i}$.

For any subset $\Gamma$ of a real vector space define the cone spanned by $\Gamma$ to be $\{\sum_{i\in I}k_iv_i|v_i\in \Gamma\cup\{0\}, k_i\in \mathbb{R}_{\geq 0}, |I|<\infty\}$ and denote it by $\cone(\Gamma)$.
A subset $\Gamma$ of $\Phi^+$ is called separable if the $\cone(\Gamma)\cap\cone(\Phi^+\backslash \Gamma)=\{0\}$. A separable set is 2 closure and cone closure biclosed in $\Phi^+$.

To end this subsection, we make a remark on the differences between the 2 closure biclosed sets and the cone closure biclosed sets (in $\Phi^+$) and on the role played by the notion of  (2 closure and cone closure) biclosed sets and  2 closure hemispaces in this paper. The 2 closure operator is finer than the cone closure operator in the sense that there are more closed sets under the 2 closure operator. The cone closure biclosed sets in $\Phi^+$ are all 2 closure biclosed. In \cite{edgar}, it is shown that the twisted Bruhat order can be defined for all 2 closure biclosed sets. In this paper, we twist the weak (right) order by 2 closure biclosed sets. Since the family of the cone closure biclosed sets is a subset of the family of the 2 closure biclosed sets,  cone closure biclosed sets can certainly be used to twist the weak order and therefore we also study the relationship between the poset properties of such twisted weak orders and the cone closure biclosed sets which are used. Finally, since there is a bijection between the set of 2 closure biclosed sets in $\Phi^+$ and the set of 2 closure hemispaces, we shall use   2 closure hemispaces as a tool in this paper to obtain certain properties of the twisted weak order. For more comparisons of these two closure operators see \cite{closurecmp}.

\subsection{Order And Lattice Theory}

Let $L$ be a partially ordered set. $L$ is a meet (resp. join) semilattice if
any two elements $x,y\in L$ admit a greatest lower bound
(called meet) (resp. least upper bound (called join)), denoted by $x\wedge y$ (resp. $x\vee y$). $L$ is called a complete meet (resp. join)
semilattice if any subset $A\subset L$ admits a greatest
lower bound (called meet) (resp. least upper bound (called join)), denoted by $\bigwedge A$ (resp. $\bigvee A$). If $L$ is both a meet semilattice and a
join semilattice, $L$ is called a lattice.

\subsection{Weak Order} Let $(W,S)$ be a Coxeter system and $w\in W$. $u\in W$ is called a (left) prefix of $W$ if $w$ has a reduced expression $s_1s_2\cdots s_t$ and $u=s_1s_2\cdots s_i,i\leq t$.   Define a partial order on $W$ such that $x\leq y$ if and only if $x$ is a prefix of $y$. It can be shown that $x\leq y$ if and only if $\Phi_x\subset \Phi_y.$ Such order is called the weak (right Bruhat) order. Under the weak order, $W$ is a complete meet semilattice and in particular when $W$ is finite, it is a complete lattice. For these facts see Chapter 3 of \cite{bjornerbrenti}.

\subsection{Infinite Reduced Words}

Let $(W,S)$ be a Coxeter system with $W$ being infinite.
A sequence $s_1s_2s_3\cdots, s_i\in S$ is called an infinite reduced word of $W$ if $s_1s_2\cdots s_j$ is a reduced expression for any $j$. Let $x=s_1s_2\cdots$ be an infinite reduced word. Define the inversion set $\Phi_x=\cup_{i=1}^{\infty}\Phi_{s_1s_2\cdots s_i}$. Two infinite reduced words are considered equal if their inversion sets coincide. One sees readily that the inversion set of an infinite reduced word is 2 closure and cone closure biclosed in $\Phi^+$. In fact, they are separable. See \cite{Inversion}. A prefix of an infinite reduced word $w$ is an element $u\in W$ such that $\Phi_u\subset \Phi_w.$

\subsection{Root System And Biclosed Sets For Affine Weyl Groups}\label{classificationsection}

The root system of an (irreducible) affine Weyl group $\widetilde{W}$ can be realized as the ``loop extension" of the root system of the corresponding finite irreducible Weyl group $W$. Let $V$ be a real Euclidean space with inner product $(-,-)$ and let  $\Phi\subset V$ be an irreducible crystallographic root system. Also assume that $V=\mathbb{R}\Phi$. Let $W$ be the associated finite Weyl group. Choose and fix a positive system $\Phi^+$ of $\Phi$ and let $\Delta$ be the simple system of $\Phi^+$. Let $\delta$ be an indeterminate. Define a $\mathbb{R}-$vector space $V'=V\oplus\mathbb{R}\delta$ and extend
the inner product on $V$ to $V'$ by requiring $(\delta,V')=0.$ For
$\alpha\in \Phi^+$, define
$\widehat{\alpha}=\{\alpha+n\delta|n\in \mathbb{Z}_{\geq 0}\}\subset
V'$. For $\alpha\in \Phi^-$, define
$\widehat{\alpha}=\{\alpha+(n+1)\delta|n\in \mathbb{Z}_{\geq
0}\}\subset V'$. For a set $\Gamma\subset \Phi$, define
$\widehat{\Gamma}=\bigcup_{\alpha\in\Gamma}\widehat{\alpha}\subset
V'$. Then the set of roots $\widetilde{\Phi}$ is $\widehat{\Phi}\uplus-\widehat{\Phi}$.
The set of positive roots is $\widetilde{\Phi}^+=\widehat{\Phi}$ and the set of negative roots is $\widetilde{\Phi}^-=-\widehat{\Phi}.$ Let $\alpha\in \widetilde{\Phi}$. Define
$$s_{\alpha}(v)=v-2\frac{(v,\alpha)}{(\alpha,\alpha)}\alpha.$$
Then we have $\widetilde{W}\simeq
\langle s_{\alpha},\alpha\in \widetilde{\Phi}\rangle$ and the set of the simple reflections is $\{s_{\alpha}|\alpha\in \Delta\}\cup \{s_{\delta-\gamma}\}$ where $\gamma$ is the highest root in $\Phi^+$.  For $v\in V$,  define the $\mathbb{R}-$linear map $t_v$ which acts on $V'$   by $t_v(u)=u+(u,v)\delta.$ For $\alpha\in \Phi,$ define $\alpha^{\vee}=2\alpha/(\alpha,\alpha)$, which is called the dual root of $\alpha$ or a coroot. Define $T$ to be the free Abelian group generated by $\{t_{\gamma^{\vee}}|\gamma\in \Delta\}$. Then it can be shown that $\widetilde{W}=W\ltimes T.$ For this construction, see \cite{Kac} or
\cite{GusDyer}.

 Any 2 closure biclosed set in $\Phi$ is of the form $\Psi^+_{\Delta_1,\Delta_2}:=(\Psi^{+}\backslash \mathbb{R}_{\geq 0}\Delta_1)\cup (\mathbb{R}\Delta_2\cap \Phi)$ where $\Psi^+$ is a positive system of $\Phi$ and $\Delta_1,\Delta_2$ are two orthogonal subsets (i.e. $(\alpha,\beta)=0$ for any $\alpha\in \Delta_1,\beta\in \Delta_2$) of the simple system of $\Psi^+.$ For a proof see \cite{biclosedphi}.
In \cite{DyerInitAffine}, Dyer classified all 2 closure biclosed sets in $\widetilde{\Phi}^+$. He shows that
any biclosed set in $\widetilde{\Phi}^+$ differs by only finite many roots from a 2 closure biclosed set of the form $\widehat{\Psi^+_{\Delta_1,\Delta_2}}$. Indeed it is  $w\cdot \widehat{\Psi^+_{\Delta_1,\Delta_2}}$ for some  $w\in W'<\widetilde{W}$ where $W'$ is the reflection subgroup generated by $\{s_{\alpha}|\alpha\in \widehat{\Delta_1\cup \Delta_2}\}$. (An essential step of the proof is to observe that for a biclosed set $B$ in $\widehat{\Phi}$, the set $I_{B}=\{\alpha\in \Phi|\widehat{\alpha}\cap B$ is infinite$\}$ is in fact a biclosed set in $\Phi$.)

\section{Various Definitions And Their Equivalence}

In this and next section unless specifically stated, all closures under consideration are assumed to be 2 closure and thus we omit the word 2 closure when discussing biclosed sets and hemispaces for convenience.

\begin{definition}
Let $(W,S)$ be a Coxeter group and let $B$ be a biclosed set in $\Phi^+$.
The (right) $B$-twisted length function $l_B: W\rightarrow \mathbb{Z}$ is defined by $l_B(w)=l(w)-2|\Phi_w\cap B|$ where $w\in W$ and $l$ is the usual length function.
Define a partial order $\preceq_B'$ on $W$ such that
$x\preceq_B'y$ if and only if $y=xt_1t_2\cdots t_l$ and $l_B(xt_1t_2\cdots t_{i-1})<l_B(xt_1t_2\cdots t_{i})$ for all $1\leq i\leq l$ where $t_i\in T.$ Such an order is called a $B-$twisted Bruhat order.
\end{definition}

See \cite{quotient}, \cite{DyHS2} for basic combinatorics of the twisted Bruhat order.
Replacing $T$ in this definition with $S$, we can define the $B-$twisted weak right order. Thus we have the following

\begin{definition}
Let $(W,S)$ be a Coxeter system. $B$ is a biclosed set
in $\Phi^+$.  We define for $x,y\in W$, $x\lhd y$ if and only if $y=xs$ for some $s\in S$ and $l_B(y)>l_B(x).$

Now define $u\leq_B v$ if one can find $u_1,u_2,\cdots,u_t$ such
that
$$u=u_1\lhd u_2\lhd u_3\lhd \cdots \lhd u_t=v.$$
Then $\leq_B$ is called the $B-$twisted weak (right) order.
\end{definition}

One notes that the covering relation in this definition can be restated using the ordinary length function.
For $s\in S$, let $\alpha_s$ be the corresponding simple root. For $x,y\in W$, $x\lhd y$ if and only if one of the following two
holds:

(1) $y=xs, l(y)=l(x)+1, s\in S$ and $x(\alpha_s)\not\in B$

(2) $x=ys, l(x)=l(y)+1, s\in S$ and $y(\alpha_s)\in B$

One routinely verifies that this is a preorder. One can see it is a well-defined partial order from its relationship with the corresponding $B-$twisted Bruhat order since $x\leq_B y$ implies $x\preceq_B' y$.
Note that $\emptyset-$twisted weak right order is the ordinary weak right order.
Our next goal is to give an alternative characterization of the twisted weak order, which is easier to work with. And by such  identification one can also see that this preorder is indeed a partial order.

\begin{definition}
Let $(W,S)$ be a Coxeter system. $B$ is a biclosed set
in $\Phi^+$.  Define the partial order $\leq_B'$ on $W$ as follow
$$x\leq_B' y \:\:\:\text{if\,and\,only\,if\,}\: \Phi_x\ominus B\subset \Phi_y\ominus B$$ where $\ominus$ denotes the symmetric difference of sets.
\end{definition}

One readily verifies $\leq_B'$ is a partial order using the fact that $\Phi_x=\Phi_y$ implies $x=y$ for $x,y\in W$.
Note that $\leq_{\Phi^+\ominus B}'$ is the reverse order of $\leq_B'$.

\begin{lemma}\label{lem:symmdiff}
$x\leq_B' y$ if and only if $\Phi_{x}\backslash \Phi_y\subset B \:\text{and}\: \Phi_y\backslash \Phi_x\subset \Phi^+\backslash B.$
\end{lemma}

\begin{proof}
The equivalence can be easily verified by inspection of a suitable Venn diagram.
\end{proof}

\begin{definition}
Let $B$ be a biclosed set in $\Phi^+.$
Define $\Sigma_B=-B\uplus (\Phi^+\backslash B)$
which is a hemispace. For $x\in W$, also define the
hemispace $\Sigma_x=\Phi_x\uplus -(\Phi^+\backslash
\Phi_x)$. Then we obtain a biclosed set $\Sigma_{B,x}=\Sigma_x\cap
\Sigma_B$ in $\Sigma_B$.
\end{definition}

\begin{lemma}
Under the map $x\mapsto \Sigma_{B,x}$, the
order $(W,\leq_B')$ embeds as a sub-poset of
$\mathscr{B}(\Sigma_{B})$, the poset of biclosed sets in $\Sigma_B$ under inclusion. If $B$ is finite, then the image of this map is precisely the set of all finite biclosed sets in $\Sigma_B.$
\end{lemma}

\begin{proof}
We first show that such a map is injective. Suppose
$\Sigma_{B,x}=\Sigma_{B,y}$. Then we have $\Phi_x\cap
(\Phi^+\backslash B)=\Phi_y\cap (\Phi^+\backslash B),
(\Phi^+\backslash \Phi_x)\cap B=(\Phi^+\backslash \Phi_y)\cap
B$. This is equivalent to $\Phi_x\backslash B=\Phi_y\backslash B$ and $B\backslash
\Phi_x=B\backslash \Phi_y$. So $B\ominus\Phi_x=B\ominus\Phi_y$. By Lemma \ref{lem:symmdiff}. this is
equivalent to $x\leq_B' y$ and $y\leq_B' x$ and thus $x=y.$ It is similar to see that
this is an order-preserving map if we note that
$$\Phi_y\cap (\Phi^+\backslash B)\subset \Phi_x\cap
(\Phi^+\backslash B)$$ is equivalent to
$$\Phi_y\backslash \Phi_x\subset B$$ and that $$(\Phi^+\backslash \Phi_y)\cap
B\subset (\Phi^+\backslash \Phi_x)\cap B$$ is equivalent to
$$\Phi_x\backslash \Phi_y\subset \Phi^+\backslash
B.$$
Now suppose $B$ is finite. Clearly the image of the map consists of the finite biclosed sets in $\Sigma_B.$
Otherwise a finite biclosed set of $\Sigma_B$ must be of the form $\Sigma_B\cap (C\uplus -(\Phi^+\backslash C))$ where $C$ is a biclosed set in $\Phi^+.$ If $C$ is infinite, then $C\cap (\Phi^+\backslash B)$ must be infinite, which is a contradiction.
\end{proof}

\begin{lemma}\label{lem:conjfin}
Let $u\in W$. Then $\mathscr{B}(u\Phi^+)$ is isomorphic to
$\mathscr{B}(\Phi^+)$ as poset.
\end{lemma}

\begin{proof}
Map $A\mapsto uA$ and we get the poset isomorphism.
\end{proof}

\begin{remark}
Note that $u\Phi^+=-\Phi_u\uplus
(\Phi^+\backslash\Phi_u)$.
The consequence of the above two lemmas is that $(W,\leq_{\emptyset}')\simeq
(W,\leq_{\Phi_x}')$ for any $x\in W$ as poset.

Therefore if $B$ is finite, $(W,\leq_B')$ is a complete meet semilattice.
If $B$ is co-finite, $(W,\leq_B')$ is a complete join semilattice.
\end{remark}

\begin{definition}
Let $\Psi$ be a hemispace and $G(\Psi)$ be the undirected graph with the vertex set
$\mathscr{B}(\Psi)$ (the set of biclosed sets in $\Psi$) and the edge set $\{(B_1,B_2)|$ if $B_1\subseteq B_2$ and
there is no $B_3\in \mathscr{B}(\Psi)$ such that $B_1\subseteq B_3\subseteq
B_2$ or $B_2\subseteq B_1$ and there is no $B_3\in \mathscr{B}(\Psi)$ such
that $B_2\subseteq B_3\subseteq B_1\}$.
\end{definition}

Such a graph has the similar flavor of the tope graph of an oriented matroid. But one notes that even if one replaces the hemispaces and the biclosed sets in the above definition with the cone closure counterparts it is
in general not the tope graph as there are no bijections between the set of cone closure hemispaces and the cone closure biclosed sets in a particular hemispace. And we also remark that when $\Phi$ is infinite, such a graph is not connected.

\begin{lemma}\label{lengthpath}
Let $\Phi_x$ and $\Phi_y$ be two finite biclosed sets
in $\Phi^+$. Then there is a path of length $|\Phi_x\ominus\Phi_y|$
connecting $\Phi_x$ and $\Phi_y$ in $G(\Phi^+)$ where $\ominus$ denotes the
symmetric difference.
\end{lemma}

\begin{proof}
Consider the hemispace $-\Phi_x\uplus
(\Phi^+\backslash\Phi_x)=x\Phi^+.$ Now the following $C_1, C_2$ are
two finite biclosed sets in this hemispace.

$$C_1=(-\Phi_x\uplus (\Phi^+\backslash \Phi_x))\cap (\Phi_x\uplus (-(\Phi^+\backslash \Phi_x)))=\emptyset$$
$$C_2=(-\Phi_x\uplus (\Phi^+\backslash \Phi_x))\cap (\Phi_y\uplus (-(\Phi^+\backslash \Phi_y)))$$
$$=(\Phi_y\backslash (\Phi_x\cap \Phi_y))\uplus -(\Phi_x\backslash (\Phi_x\cap \Phi_y))$$

Denote $m=|\Phi_x\ominus\Phi_y|$.

Since $\mathscr{B}(x\Phi^+)$ is isomorphic to $\mathscr{B}(\Phi^+)$ as poset. Then we
can find $\alpha_1, \alpha_2, \cdots, \alpha_m\in C_2$ such that
$$\emptyset=C_1\subset \{\alpha_1\} \subset \{\alpha_1,\alpha_2\}\subset \cdots  \subset \{\alpha_1,\alpha_2,\cdots, \alpha_m\}=C_2$$
with each set in this chain biclosed in the hemispace $x\Phi^+$.

Now consider the following hemispaces:
$$\Phi_x\uplus -(\Phi^+\backslash \Phi_x)$$
$$(\Phi_x\uplus -(\Phi^+\backslash \Phi_x)\backslash \{-\alpha_1\})\cup \{\alpha_1\}$$
$$(\Phi_x\uplus -(\Phi^+\backslash \Phi_x)\backslash \{-\alpha_1,-\alpha_2\})\cup \{\alpha_1,\alpha_2\}$$
$$\cdots$$
$$(\Phi_x\uplus -(\Phi^+\backslash \Phi_x)\backslash \{-\alpha_1,-\alpha_2,\cdots,-\alpha_m\})\cup \{\alpha_1,\alpha_2,\cdots,\alpha_m\}$$
(the last one is $\Phi_y\uplus -(\Phi^+\backslash \Phi_y)$.)

Intersecting these hemispaces with $\Phi^+$ one gets
$$B_1=\Phi_x$$
$$B_2=B_1\uplus \{\alpha_1\} \:\text{or}\: B_1\backslash \{-\alpha_1\}$$
$$B_3=B_2\uplus \{\alpha_2\} \:\text{or}\: B_2\backslash \{-\alpha_2\}$$
$$\cdots$$
$$B_m=B_{m-1}\uplus \{\alpha_{m-1}\} \:\text{or}\: B_{m-1}\backslash \{-\alpha_{m-1}\}=\Phi_y$$

proving the lemma.
\end{proof}

\begin{theorem}
$(W,\leq_B')$ and $(W,\leq_B)$ define the same order.
\end{theorem}

\begin{proof}
It is immediate that $x\leq_B y$ implies $x\leq_B' y.$ Now assume $x\leq_B' y$.
By Lemma \ref{lengthpath} one can find a path of minimal
possible length $|\Phi_x\ominus\Phi_y|$ in $G(\Phi^+)$ connecting
$\Phi_x$ and $\Phi_y$. Hence along this path, one either removes a
root in $\Phi_x$ but not in $\Phi_y$ or add a root in $\Phi_y$ but
not in $\Phi_x$, the former is in $B$ and the latter is not in $B$.
Suppose that the path goes through $\Phi_{x}, \Phi_{u_1}, \Phi_{u_2}, \cdots, \Phi_{u_t}, \Phi_y$.
We have $x\lhd u_1\lhd u_2 \lhd \cdots \lhd u_t \lhd y.$
\end{proof}

Therefore in what follows we shall freely use the identification of these two descriptions of the twisted weak order.
In the remaining part of this section we give some examples of the twisted weak orders and also indicate some situations in which twisted weak orders show up.

\begin{example}
(\emph{Limit Weak Order}) Denote by $W_l$ the set of all infinite reduced words.
It is natural to extend the weak order to the set  $W\cup W_l$, i.e. two words $x\leq y$ if and only if $\Phi_x\subset \Phi_y.$
When restricted to $W_l$ such an order is called the limit weak order in \cite{titsboundary}. Two infinite reduced words are said to be in the same block if their inversion sets differ only by finite many roots. In \cite{titsboundary}, the authors showed that when $W$ is finite rank and  word hyperbolic, each block of the limit
weak order is isomorphic to the weak order of some finite Coxeter group. If $W$ is an affine Weyl group, each block of  the limit weak order is isomorphic to the weak order of some affine Coxeter group (or consists of just one element). But if infinite rank Coxeter groups are considered, a block of limit weak order can be isomorphic to the nontrivial twisted weak order on some Coxeter group as demonstrated by the following easy example. Let $(W_i,S_i), i=1,2,\cdots$ be the Coxeter systems of type $A_1$ (i.e. isomorphic to $\mathbb{Z}_2$) and $S_i=\{s_i\}$. Consider the weak direct product $W:={}_w\Pi_{i=1}^{\infty}W_i$ (the subgroup of the direct product $\Pi_{i=1}^{\infty}W_i$ generated by the (isomorphic) subgroups $\{W_i\}_{i=1}^{\infty}$). Then $(W,S)$ is a Coxeter system. Any infinite subset of its positive system $\Phi^+=\{\alpha_{s_i}|i=1,2,\cdots\}$ is biclosed and in fact the inversion set of an infinite reduced word. We shall show that for any infinite biclosed set $B\subset \Phi^+$, there exists a block in the limit weak order isomorphic to $(W,\leq_B)$.  Such a block is the one containing the infinite reduced word $\Pi_{\alpha_{s_i}\in B}s_i.$ The isomorphism is given by the map: $w\mapsto w\Pi_{\alpha_{s_i}\in B}s_i$ and one checks readily that $w_1\leq_B w_2\Leftrightarrow \Phi_{w_1\Pi_{\alpha_{s_i}\in B}s_i}\subset \Phi_{w_2\Pi_{\alpha_{s_i}\in B}s_i}$.
\end{example}

\begin{example}
(\emph{1-skeleton of The Tessellation by Permutahedra})
In this example, one can see that the twisted weak order may have a very different global structure from the ordinary weak order.
For an (irreducible) affine Weyl group $\widetilde{W}$, the Hasse diagram of certain twisted weak order gives the ``1 skeleton" of the (cell complex of the) tessellation of the Euclidean space by  permutahedra of the finite parabolic subgroups of $\widetilde{W}.$ Such orders are defined by biclosed sets of the form $\widehat{\Phi^+}$ for some positive system $\Phi^+$ of the corresponding finite Weyl group $W$ (See Corollary \ref{lattice} below) and is in fact a (non-complete) lattice. (Part of) the Hasse diagrams of the twisted weak orders $\leq_{\widehat{\Phi^-}}$ on groups of type $\widetilde{A}_1,\widetilde{A}_2$ are depicted below.

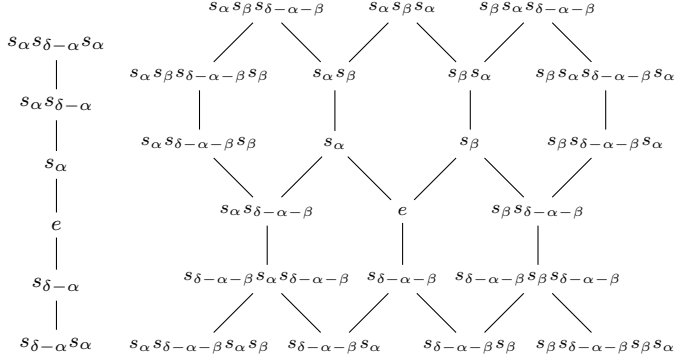
\begin{tikzpicture}[scale=.4]
\footnotesize
\node (o) at (6,2)   {$s_{\alpha}s_{\delta-\alpha}s_{\alpha}$};
  \node (k)  at (6,0)  {$s_{\alpha}s_{\delta-\alpha}$};
  \node (h) at (6,-2)  {$s_{\alpha}$};
  \node (g) at (6,-4) {$e$};
  \node (l) at (6,-6)  {$s_{\delta-\alpha}$};
  \node (m)  at (6,-8)  {$s_{\delta-\alpha}s_{\alpha}$};
  \draw  (h)--(g);
  \draw  (h)--(k);
  \draw  (l)--(g);
  \draw  (m)--(l);
  \draw  (k)--(o);
  \end{tikzpicture}
\begin{tikzpicture}[scale=.45]
\scriptsize
  \node (a) at (0,2) {$s_{\alpha}s_{\beta}s_{\alpha}$};
   \node (b) at (4,2) {$s_{\beta}s_{\alpha}s_{\delta-\alpha-\beta}$};
  \node (c) at (-2,0) {$s_{\alpha}s_{\beta}$};
  \node (d) at (2,0) {$s_{\beta}s_{\alpha}$};
  \node (e) at (-2,-2) {$s_{\alpha}$};
  \node (f) at (2,-2) {$s_{\beta}$};
  \node (g) at (0,-4) {$e$};
  \node (h)  at (4,-4) {$s_{\beta}s_{\delta-\alpha-\beta}$};
  \node (i)   at (6,-2)  {$s_{\beta}s_{\delta-\alpha-\beta}s_{\alpha}$};
  \node (j)   at (6,0)  {$s_{\beta}s_{\alpha}s_{\delta-\alpha-\beta}s_{\alpha}$};
  \node (k)  at (-4,-4) {$s_{\alpha}s_{\delta-\alpha-\beta}$};
  \node (l)  at (-6,-2) {$s_{\alpha}s_{\delta-\alpha-\beta}s_{\beta}$};
  \node (m)  at (-6,0) {$s_{\alpha}s_{\beta}s_{\delta-\alpha-\beta}s_{\beta}$};
  \node  (n)  at  (-4,2) {$s_{\alpha}s_{\beta}s_{\delta-\alpha-\beta}$};
  \node   (o)  at  (0,-6) {$s_{\delta-\alpha-\beta}$};
  \node  (p)  at  (-2,-8)  {$s_{\delta-\alpha-\beta}s_{\alpha}$};
  \node  (q)  at  (-4,-6)  {$s_{\delta-\alpha-\beta}s_{\alpha}s_{\delta-\alpha-\beta}$};
  \node  (r)  at  (2,-8)  {$s_{\delta-\alpha-\beta}s_{\beta}$};
  \node  (t)  at  (4,-6)  {$s_{\delta-\alpha-\beta}s_{\beta}s_{\delta-\alpha-\beta}$};
  \node  (u)  at  (-6,-8) {$s_{\alpha}s_{\delta-\alpha-\beta}s_{\alpha}s_{\beta}$};
  \node  (v)  at  (6,-8) {$s_{\beta}s_{\delta-\alpha-\beta}s_{\beta}s_{\alpha}$};
  \draw (c) -- (a);
  \draw (d) -- (a);
  \draw (d) -- (b);
  \draw (e)--(c);
  \draw  (f)--(d);
  \draw  (g)--(e);
  \draw  (g)--(f);
  \draw  (h)--(f);
  \draw   (h)--(i);
  \draw   (i)--(j);
  \draw   (j)--(b);
  \draw   (k)--(e);
  \draw   (k)--(l);
  \draw   (l)--(m);
  \draw   (m)--(n);
  \draw  (c)--(n);
  \draw  (o)--(g);
  \draw   (p)--(o);
  \draw   (p)--(q);
  \draw   (q)--(k);
  \draw  (r)--(o);
  \draw  (r)--(t);
  \draw  (t)--(h);
  \draw  (u)--(q);
  \draw  (v)--(t);
\end{tikzpicture}
\end{example}

\begin{example}
(\emph{The Poset of 2 closure Biclosed Sets in $\Phi^+$})  In \cite{DyerInitAffine}, Dyer classified all 2 closure biclosed sets in $\widetilde{\Phi}^+$ for affine Weyl groups and described the partial order of them under inclusion. Similarly one can consider the blocks in such orders. Within one block two biclosed sets differ by only finite many roots.  It follows from Dyer's results that every block is isomorphic to the twisted weak of some Coxeter groups. To be precise, the block containing the biclosed sets $\widehat{\Psi_{\Delta_1,\Delta_2}^+}$ is isomorphic to the twisted weak order $<_{\Phi^+_{\widetilde{W_2}}}$ on $\widetilde{W_1}\times \widetilde{W_2}$ where $\widetilde{W_i}, i=1,2$ is the reflection subgroup of $\widetilde{W}$ generated by $s_{\alpha},\alpha\in\widehat{\Delta_i}, i=1,2$.
\end{example}

\section{Characterization of The Inversion Sets of Infinite Reduced Words}

In this section we will prove the main result of this paper regarding the geometry of the root system and the combinatorics of the twisted weak orders. In what follows $(W,S)$ is a Coxeter system. $\Phi$ is its root system and $\Phi^+$ is the set of positive roots. $B$ is a biclosed set in $\Phi^+.$

\begin{lemma}\label{intervallemma}
Suppose $x\leq_B y\leq_B z$.
Then $\Phi_x\backslash \Phi_y\subset \Phi_x\backslash \Phi_z$,
$\Phi_y\backslash \Phi_x\subset \Phi_z\backslash \Phi_x,$
$\Phi_y\backslash \Phi_z\subset \Phi_x\backslash \Phi_z$ and
$\Phi_z\backslash \Phi_y\subset \Phi_z\backslash \Phi_x.$

Consequently $\Phi_x\backslash \Phi_z=(\Phi_x\backslash
\Phi_y)\uplus (\Phi_y\backslash \Phi_z)$ and $\Phi_z\backslash
\Phi_x=(\Phi_z\backslash \Phi_y)\uplus (\Phi_y\backslash
\Phi_x)$.
\end{lemma}

\begin{proof}
Take $\alpha\in \Phi_x\backslash \Phi_y$. Then
$\alpha\in B$. If $\alpha\in \Phi_z$ then $\alpha\in
\Phi_z\backslash \Phi_y$ then $\alpha\in \Phi^+\backslash B$. This
is a contradiction. So $\alpha\in \Phi_x\backslash \Phi_z.$ The
other assertions can be proved similarly.
\end{proof}

Let $B$ be a biclosed set in $\Phi^+$ and $x,y\in W$. Denote by $[x,y]_B$ the set $\{w\in W|x\leq_B w\leq_B y\}.$ The following Corollary shows that locally the intervals of a twisted weak order share the same structure of the ordinary weak order, though globally those orders could be very different.

\begin{corollary}\label{invariant}
Suppose $x\leq_B y$. Then
$[x,y]_B=[x,y]_{\Phi_x}=[x,y]_{\Phi^+\backslash \Phi_y}.$
Consequently for any $B$ such that $x\leq_B y$, $[x,y]_B$ is
independent of $B$. In addition these intervals are isomorphic as posets.
\end{corollary}

\begin{proof}
It is evident
from the definition that $x\leq_{\Phi_x} y$ for all $y\in W$. Take $z\in [x,y]_B$. We have $\Phi_z\backslash
\Phi_y\subset \Phi_x\backslash \Phi_y\subset \Phi_x$ by Lemma \ref{intervallemma}.
We also have $\Phi_y\backslash \Phi_z\subset \Phi_y\backslash \Phi_x\subset
\Phi^+\backslash\Phi_x$  by Lemma \ref{intervallemma}. Hence $z\leq_{\Phi_x} y$.
This shows $[x,y]_B\subset[x,y]_{\Phi_x}$. Conversely take $z\in
[x,y]_{\Phi_x}$. We have $\Phi_z\backslash \Phi_y\subset \Phi_x\backslash
\Phi_y\subset B$ by Lemma \ref{intervallemma}. We also have $\Phi_y\backslash \Phi_z\subset
\Phi_y\backslash \Phi_x\subset \Phi^+\backslash B$  by Lemma \ref{intervallemma}. Hence $z\leq_{B} y$. Similarly one can show $x\leq_B z$. This
shows $[x,y]_B\supset[x,y]_{\Phi_x}$.
It can be proved in the same
manner that $[x,y]_B=[x,y]_{\Phi^+\backslash \Phi_y}$.

Now assume $x\leq_B u\leq_B v\leq_B y$. Then by Lemma \ref{intervallemma}, $\Phi_v\backslash
\Phi_u\subset \Phi_v\backslash \Phi_x\subset \Phi^+\backslash
\Phi_x$. Similarly $\Phi_u\backslash \Phi_v\subset \Phi_x\backslash
\Phi_v\subset \Phi_x$. Hence $u\leq_{\Phi_x} v$. Assume
$x\leq_{\Phi_x}u\leq_{\Phi_x}v\leq_{\Phi_x}y$. Then $\Phi_u\backslash
\Phi_v\subset \Phi_x\backslash \Phi_v\subset \Phi_x\backslash
\Phi_y\subset B$ and $\Phi_v\backslash \Phi_u\subset \Phi_v\backslash
\Phi_x\subset \Phi_y\backslash \Phi_x\subset \Phi^+\backslash B$. So
$u\leq_B v.$ This implies $[x,y]_B$ and $[x,y]_{\Phi_x}$ are
isomorphic as poset. The other isomorphism can be proved in the
similar way.
\end{proof}

\begin{proposition}
Let $A\subset W$ and let $B$
be a  biclosed set in $\Phi^+$. If there exists $u\in W$ such that $u\leq_B
x$ for all $x\in A$ then $\bigcap_{x\in A}[u,x]_B$ admits a unique
maximal element. If there exists $v\in W$ such that $v\geq_B x$ for
all $x\in A$ then $\bigcap_{x\in A}[x,v]_B$ admits a unique minimal
element.
\end{proposition}

\begin{proof}
We prove the first assertion.
$\bigcap_{x\in A}[u,x]_B\simeq\bigcap_{x\in
A}[u,x]_{\Phi_u}$ by Corollary \ref{invariant}. But $(W,\leq_{\Phi_u})$ is
a complete meet semilattice and we obtain the desired assertion. The
other assertion can be proved similarly.
\end{proof}

\begin{lemma}\label{lowerbound}
Take $B=\Phi_w, w\in W_l$ (the set of infinite
reduced words). Then for $x,y\in W$ one can find $z\in W$ such that $z$
is a lower bound of $x,y$ in $(W,\leq_B)$.
\end{lemma}

\begin{proof}
$\Phi_x\cap \Phi_w$ and $\Phi_y\cap \Phi_w$ are
both finite. So one can find a prefix $z$ of $w$ such that
$\Phi_z\supset (\Phi_x\cap \Phi_w)\cup (\Phi_y\cap
\Phi_w)$. Clearly $\Phi_z\backslash \Phi_x$ and $\Phi_z\backslash
\Phi_y$ are both in $\Phi_w.$ $\Phi_x\backslash \Phi_z\subset
\Phi_x\backslash (\Phi_x\cap \Phi_w)=\Phi_x\backslash
\Phi_w\subset \Phi^+\backslash \Phi_w$. Hence $x\geq_B z$
and similarly $y\geq_B z$.
\end{proof}

\begin{theorem}\label{twistsemilattice}
Take $w\in W_l$ (the set of infinite reduced words). Then
$(W,\leq_{\Phi_w})$ is a non-complete meet semilattice and $(W,\leq_{\Phi^+\backslash\Phi_w})$
is a non-complete join semilattice.
\end{theorem}

\begin{proof}
We prove the first assertion.
Let $x,y\in W$. Suppose that $\{u|u\leq_{\Phi_w} x, u\leq_{\Phi_w} y\}$ has two
different maximal elements $s,t$. (This set is not empty by Lemma \ref{lowerbound}.) Then $s,t$ has a lower bound $v$ by Lemma \ref{lowerbound} again. By corollary \ref{invariant}
$[v,x]_{\Phi_w}\cap [v,y]_{\Phi_w}$ is isomorphic to
$[v,x]_{\Phi_v}\cap [v,y]_{\Phi_v}$. But in the latter $s,t$ have
an upper bound as $(W,\leq_{\Phi_v})$ is a complete meet semilattice and then so do in the former. This is a
contradiction.
To see that such a meet semilattice is not complete, let $w=s_1s_2s_3\cdots$. Then
the chain $s_1\geq_{\Phi_{w}}s_1s_2\geq_{\Phi_{w}}s_1s_2s_3\geq \cdots$ does not have a lower bound.
The second dual assertion can be proved similarly.
\end{proof}

In the remaining part of this section we shall show that for an (irreducible) affine Weyl group, the converse of Theorem \ref{twistsemilattice}, i.e. if $\leq_B$ is a meet semilattice then $B=\Phi_w$ with $w\in W\cup W_l$, also holds. In what follows an affine Weyl group is assumed to be an irreducible one. Along the way, we also give a classification of the 2 closure biclosed sets that come from infinite reduced words for affine Weyl groups. The proof given here is shorter than the one given in \cite{wang}.

An element $w$ of a Coxeter group is said to be straight if $l(w^n)=|nl(w)|$ for all $n\ge 1$.
Let $w$ be a straight element. Then $w^{\infty}:=www\cdots$ is a well-defined infinite reduced word.

Now let $W$ be an irreducible Weyl group with the root system $\Phi$ and $V=\mathbb{R}\Phi$.
A vector $\lambda\in V$ is called a coweight if $(\lambda,\alpha)\in \mathbb{Z}$ for all $\alpha\in \Phi.$ The set of coweights is called the coweight lattice. This is a free Abelian group. Its subgroup generated by the coroots is called the coroot lattice. The index of the coroot lattice in the coweight lattice is finite and called the connection index.

\begin{lemma}\label{infinitewordsform}
Let $\gamma$ be an element in the coroot lattice. Then $t_{\gamma}$ is straight. Moreover $\Phi_{t_{\gamma}^{\infty}}=\widehat{\Psi^+_{L,\emptyset}}$ for some biclosed set $\Psi^+_{L,\emptyset}$ in $\Phi.$
\end{lemma}

\begin{proof}
Take $\beta+n\delta\in \widehat{\Phi}$. Then
$$t_{-\gamma}^k(\beta+n\delta)=\beta+(n+k(\beta,-\gamma))\delta.$$
This is a negative root if
$$n<k(\beta,\gamma)$$
or $n=k(\beta,\gamma)$ and $\beta\in \Phi^-.$

If $(\beta,\gamma) < 0$, then the action of $t_{\gamma}$ will increase the
coefficient of $\delta$. So $\beta+n\delta$ will never be made negative by any power of $t_{\gamma}$.

If $(\beta,\gamma)=0$, then $t_{\gamma}$ will stabilize $\beta+n\delta$. So the power of $t_{\gamma}$ will never
make $\beta+n\delta$ negative.

If $(\beta,\gamma)>0$, then eventually $n < k(\beta,\gamma)$ when $k$ is sufficiently large. Then $\beta+n\delta$ is eventually made negative.

So we have $\Phi_{t_{\gamma}^k}\subset \Phi_{t_{\gamma}^{k+1}}$
for all $k\geq 1$, which shows $t_{\gamma}$ is straight.

By the above, $\Phi_{t_{\gamma}^{\infty}}=\widehat{\{\zeta|\zeta\in \Phi, (\zeta,\gamma)>0\}}$.  It is easy to see that
$\{\zeta|\zeta\in \Phi, (\zeta,\gamma)>0\}$ is biclosed in $\Phi.$  But $\{\zeta|\zeta\in \Phi, (\zeta,\gamma)>0\}$ cannot be of the form
$\Psi^+_{L,M}$ where $M\neq \emptyset$. If not, take $\lambda$ and $-\lambda$ such that $\lambda\in M$ their inner products with $\gamma$
cannot be both positive.
\end{proof}

\begin{lemma}\label{infinitewordsformtwo}
Any biclosed set $\widehat{\Psi^+_{L,\emptyset}}$ in $\widehat{\Phi}$ is of the form $\Phi_{t_{\gamma}^{\infty}}$ for some $ \gamma$ in the coroot lattice.
\end{lemma}

\begin{proof}
Pick $\gamma'$ in the coweight lattice such that $(\gamma',\alpha)=0$ for all
$\alpha\in L$ and $(\gamma',\alpha)>0$ for all $\alpha\in \Delta\backslash L$ where $\Delta$ is the simple system in $\Psi$. Let $n$ be
the connection index and $\gamma=n\gamma'$. Then $t_{\gamma}$ has the desired property by Lemma \ref{infinitewordsform}.
\end{proof}

\begin{proposition}\label{infinitewordschar}
A biclosed set in the set of positive roots $\widehat{\Phi}$ of an affine Weyl group $\widetilde{W}$ is the inversion set of an element in the group or an infinite reduced word if and only if it is of the form $w\cdot \widehat{\Psi^+_{\Delta_1,\emptyset}}$ where $\Psi^+$ is a positive system of $\Phi$  and $\Delta_1$ is a subset of the simple system of $\Psi^+.$
\end{proposition}

\begin{proof}
It follows from Lemma \ref{infinitewordsformtwo} and the easy fact  $w\cdot\Phi_u=\Phi_{wu}, w\in \widetilde{W}$ that $w\cdot \widehat{\Psi^+_{\Delta_1,\emptyset}}$ is  an inversion set. Now because of the classification of the biclosed sets in $\widehat{\Phi}$, we only have to show that $w\cdot \widehat{\Psi^+_{\Delta_1,\Delta_2}}$ with $\Delta_2\neq \emptyset$ is not from an inversion set. Again because $w\cdot\Phi_u=\Phi_{wu}$, it suffices to show that only for $\widehat{\Psi^+_{\Delta_1,\Delta_2}}$. Suppose to the contrary that $\widehat{\Psi_{\Delta_1,\Delta_2}^+}=\Phi_w$ where $w$ is an infinite reduced
word or a group element. Assume $\alpha\in \Delta_2$. Then $\alpha+m\delta$ and $-\alpha+n\delta$
are both in $\widehat{\Psi^+_{\Delta_1,\Delta_2}}$ for some $m,n\in
\mathbb{Z}_{\geq 1}$. Then they must be both in some $\Phi_u, u\in
\widetilde{W}$. Because $\Phi_u$ is biclosed so we have
$$2(\alpha+m\delta)+(-\alpha+n\delta)=\alpha+(2m+n)\delta\in \Phi_u$$
$$3(\alpha+m\delta)+2(-\alpha+n\delta)=\alpha+(3m+2n)\delta\in \Phi_u$$
$$\cdots$$

This implies $\Phi_u$ is infinite which is impossible.
\end{proof}

\begin{remark}
The above proposition indeed shows that if a biclosed set $B$ in $\widehat{\Phi}$ is not an inversion set, then there exists some $\alpha+k\delta$ and $-\alpha+t\delta, k,t>0,\alpha\in \Phi$ both contained in $B$. This is clear if $B$ is of the form $\widehat{\Psi^+_{\Delta_1,\Delta_2}}$ (take $\alpha\in \Delta_2$). Otherwise $B=w\cdot \widehat{\Psi^+_{\Delta_1,\Delta_2}}$ for some $\Psi^+,\Delta_1,\Delta_2$ and $w\in W'$ where $W'$ is the reflection subgroup as described in the end of the section \ref{classificationsection}. In this case $|B\ominus\widehat{\Psi^+_{\Delta_1,\Delta_2}}|$ is finite and the assertion also holds.
\end{remark}

\begin{corollary}\label{lattice}
Let $\widetilde{W}$ be an affine
Weyl group (with the corresponding Weyl group $W$) and $\Psi^+$ be a positive system of $\Phi$ (the root system
of $W$). Then $(\widetilde{W},\leq_{\widehat{\Psi^+}})$ is a lattice.
\end{corollary}

\begin{proof}
Note that $\widetilde{\Phi}^+\backslash \widehat{\Psi^+}=\widehat{\Phi}\backslash \widehat{\Psi^+}=\widehat{\Psi^-}$.
Then this assertion follows from Theorem \ref{twistsemilattice} and Corollary \ref{infinitewordschar}.
\end{proof}

For an infinite Coxeter group we denote by $\overline{W}$ the union of $W$ and $W_l$ (set of infinite reduced words).

\begin{proposition}\label{nonlattice}
Let $\widetilde{W}$ be an
affine Weyl group (with the corresponding Weyl group $W$). If $B$ is a biclosed set (resp. biconvex set) in $\widetilde{\Phi}^+$ which is not of the form
$\Phi_w$ where $w\in \overline{\widetilde{W}}$, then
$(\widetilde{W},\leq_B)$ is not a semilattice.
\end{proposition}

\begin{proof}
If $B$ is not of the form $\Phi_w,w\in \overline{\widetilde{W}}$,
then there exists $\alpha+k\delta$ and $-\alpha+t\delta$ both in
$B$ by the remark succeeding Proposition \ref{infinitewordschar}. Then consider $s_{\alpha+k\delta}$ and $s_{-\alpha+t\delta}$.
Suppose that they admit a lower bound $z$. Then $\alpha+k\delta$ and
$-\alpha+t\delta$ cannot be both in $\Phi_z$. So say
$\alpha+k\delta$ is not. Then it is in
$\Phi_{s_{\alpha+k\delta}}\backslash \Phi_z\subset
\widetilde{\Phi}^+\backslash B$. This is a contradiction. The similar contradiction can be obtained if one assumes that $-\alpha+t\delta$ is not in $\Phi_z.$

For the assertion in terms of the biconvex sets, one only needs to notice that all those non-biconvex 2 closure biclosed sets in $\widehat{\Phi}$  are not inversion sets.
\end{proof}

Therefore we can conclude

\begin{theorem}\label{mainth}
Let $\widetilde{W}$ be an affine
Weyl group. The following statements about a biclosed (resp. biconvex) set $B$ in $\widetilde{\Phi}^+$ are equivalent:

(1) $B$ is the inversion set of an
infinite reduced word;

(2) $(\widetilde{W},\leq_B)$ is a
non-complete meet semilattice.
\end{theorem}

\begin{proof}
The equivalence of (1) and (2) follows from Proposition \ref{nonlattice} and Theorem \ref{twistsemilattice}.
\end{proof}

\begin{remark}
A corollary of the above theorem is that for affine Weyl groups, an infinite biclosed set $B$ in $\widehat{\Phi}$ is separable if and only if $\leq_B$ is either a non-complete meet semilattice or a non-complete join semilattice. To see this, one notes that if both $B$ and $\widehat{\Phi}\backslash B$ are not an inversion set then $\delta$ is in the cone closure of both $B$ and $\widehat{\Phi}\backslash B.$ Thus $B$ cannot be separable.
\end{remark}

\begin{remark}
Let $(W_i,S_i), i=1,2$ be two Coxeter systems with the root system and the positive system $\Phi_i, \Phi^+_i$ respectively. Let $W=W_1\times W_2, S=S_1\uplus S_2$. Then the Coxeter system $(W,S)$ has the root system $\Phi=\Phi_1\uplus \Phi_2$ and the positive system $\Phi^+=\Phi_1^+\uplus \Phi_2^+.$ Let $B$ be a biclosed set in $\Phi^+.$ Then $B=B_1\uplus B_2$ where $B_i=B\cap \Phi_i$. The twisted weak order $(W,\leq_B)$ is isomorphic to the poset product $(W_1,\leq_{B_1})\times (W_2, \leq_{B_2})$. Note that if the product poset is a meet semilattice then each one of the poset is also a meet semilattice and if each one of the poset is a meet semilattice then so is their product. Now suppose that for each $W_i$, $B_i$ is the inversion set of an infinite reduced word if and only if $(W_i,\leq_{B_i})$ is a non-complete meet semilattice. Then for $W$,  $B$ is the inversion set of an infinite reduced word if and only if $(W,\leq_{B})$ is a non-complete meet semilattice. Therefore the equivalence in Theorem \ref{mainth} holds for any affine Coxeter group.
\end{remark}

We conjecture that the same assertion holds for any arbitrary finite rank infinite Coxeter group, i.e. a biclosed set $B$ in the positive system is the inversion set of an infinite reduced word if and only if $(W,\leq_B)$ is a non-complete meet semilattice. The ``only if" part is true by above Theorem \ref{twistsemilattice}.

\section{Acknowledgement}

The author acknowledges the support from Guangdong  Natural Science Foundation  Project 2018A030313581.
Some results of the paper are based on part of the author's thesis. The author wishes to thank his advisor Matthew Dyer, who was aware of the twisted weak order while studying  the twisted Bruhat order, for his guidance. The author thanks the anonymous referees for their time and useful comments for improving the paper.

\bibliographystyle{plain}

\end{document}